\documentclass[a4paper,11pt]{article}
\usepackage[english]{babel}
\usepackage[T1]{fontenc}
\usepackage[latin1]{inputenc}
\usepackage{amsmath,amssymb,mathrsfs}
\usepackage{soul}
\usepackage{xypic}
\usepackage{amsthm}

\theoremstyle{plain}
\newtheorem{thm}{Theorem}[section]
\newtheorem{prop}[thm]{Proposition}

\theoremstyle{definition}
\newtheorem{defi}[thm]{Definition}
\newtheorem{rem}[thm]{Remark}

\newtheorem{lem}[thm]{Lemma}
\newcommand{\pdt}{\mathbin{\rule[.1ex]{.5em}{.03em}\rule[.2ex]{.03em}{.7ex}}}

\title{Invariant formulation of a variational problem}
\author{Imsatfia Moheddine}

\begin{document}

\maketitle

\begin{abstract}
In this paper we present a Hamiltonian formulation of multisymplectic type of an invariant variational problem on smooth submanifold of dimension $p$ in a smooth manifold of dimension $n$ with $p<n$.
\end{abstract}


\begin{LARGE}
 Introduction
\end{LARGE}

Let $\mathcal{M}$ be a smooth manifold of dimension $n$ and $\mathcal{N}$ be an oriented smooth submanifold of dimension $p$ in $T_x\mathcal{M}$, for $x\in\mathcal{M}$. We study the case $n=3$ and $p=2$ in details and consider generalizations. Let $\Sigma$ be an oriented surface in $\mathcal{M}$, in each $M\in\Sigma$, denote by $P$ the tangent plane. We define the \textit{Grassmannian} as the set of all these elements. We define the 2-form areolar action $\ell$ homogenous of degree zero in $y$ and we denote by $L$ a lagrangian homogenous of degree one in $y$ such that $\frac{d L}{dy}=\ell$. For some basis $(u_1,u_2)$ of $P$, and $(\theta^1,\theta^2)$ the dual basis, we define the action 
$$
 \mathcal{L}(\Sigma):=\int_\Sigma L(x,u_1\wedge u_2)\theta_1\wedge\theta_2,
$$

We use notions of mechanics and we construct the multisymplectic type of Hamiltonnian formulation. We prove that this description is conserved by graphs by comparison with the presentation of Cartan in \cite{Cartan1933}. Finally we study this for $n>3$ and $p<n$.
\section{The case presented by Cartan: a surface $ \Sigma $ in a smooth manifold $ \mathcal{M} $ of dimension 3}
Let  $\mathcal{M}$ be a smooth manifold of dimension 3. We define the Grassmannian bundle of planes in $T\mathcal{M}$ by:
\[
 Gr^2\mathcal{M}:=\{(x,E)\vert x\in\mathcal{M}, \ \ E \text{ oriented plane}\footnote{$E$ is called ``element'' by Cartan.}\text{ in } T_x\mathcal{M}\}
\]
\begin{defi} Assume that $\Lambda^2T_x\mathcal{M}:=(\Lambda^2T^\ast_x\mathcal{M})^\ast$ and $ \Lambda^2T\mathcal{M}=\bigcup_{x\in\mathcal{M}}\Lambda^2T_x\mathcal{M}$.
  Let $u_1,u_2\in T_x\mathcal{M}$. We define $u_1\wedge u_2\in \Lambda^2T_x\mathcal{M}$ by: $\forall\alpha,\beta\in T^\ast_x\mathcal{M}$
\[
 (\alpha\wedge\beta)(u_1\wedge u_2):=\alpha\wedge\beta(u_1,u_2)=\alpha(u_1)\beta(u_2) -\alpha (u_2)\beta(u_1).
\]
\end{defi}
We denote by $\pi_{Gr^2\mathcal{M}}:Gr^2\mathcal{M}\longrightarrow\mathcal{M}$ the canonical projection. Let $x=(x^1,x^2,x^3)$ be a local coordinates on $\mathcal{M}$ and $\left(\frac {\partial}{\partial x^1},\frac{\partial}{\partial x^2},\frac{\partial}{\partial x^3}\right)$ a basis of $T_x\mathcal{M}$, so a basis of $\Lambda^2T_x\mathcal{M}$ is given by
 \[
\left(\frac {\partial}{\partial x^1}\wedge\frac {\partial}{\partial x^2},\frac{\partial} {\partial x^2}\wedge\frac {\partial}{\partial x^3},\frac {\partial}{\partial x^3}\wedge \frac {\partial}{\partial x^1}\right).
\]
Let $(u_1,u_2)$ be a direct basis of $E$, and $(x^1,x^2,x^3; y^{12},y^{23},y^{31})$ the coordinates on $\Lambda^2T\mathcal{M}$ where $y^{12},y^{23}$, $y^{31}$ are the coordinates representing $E$. Then, we have
\[
 u:=u_1\wedge u_2=y^{12}\frac {\partial}{\partial x^1}\wedge\frac {\partial}{\partial x^2}+ y^{23}\frac{\partial} {\partial x^2}\wedge\frac {\partial}{\partial x^3}+ y^{31}\frac {\partial}{\partial x^3}\wedge \frac {\partial}{\partial x^1}
\]
\begin{rem}\label{rem}
  We consider the 3-form $\omega:=dx^1\wedge dx^2\wedge dx^3$. We have:
 \begin{enumerate}
\item The kernel  of $u\pdt\omega$ is the plane $E$.
\item $Gr^2_x\mathcal{M}\simeq(\Lambda^2T_x\mathcal{M}\setminus\{0\})/\mathbb{R}^\ast+$.
\end{enumerate}
\end{rem}
We deduce that the homogeneous coordinates on the Grassmannian bundle over $\mathcal{M}$ are $(x;y):=(x^1,x^2,x^3;[y^{12}:y^{23}:y^{31}])$ 
\[
  Gr^2\mathcal{M}:=\{(x;y)\vert x=(x^1,x^2,x^3)\in\mathcal{M}\text{ and } E=y=[y^{12}:y^{23}:y^{31}]\},
\]
\begin{defi}
 Let $\mathcal{M}$ be a differential manifold and $n\in\mathbb{N}$, we say that a $(n+1)$-form $\Omega$ in $\mathcal{M}$ is \textit{multisymplectic} if we have
\begin{enumerate}
 \item $\Omega$ is nondegenerate, that is, $\forall M\in\mathcal{M},\ \ \forall\xi\in T_M\mathcal{M}$, if $\xi\pdt\Omega_M=0$, then $\xi=0$).
\item $\Omega$ is closed, $d\Omega=0$.
\end{enumerate}
Any manifold $\mathcal{M}$ equiped with a multisymplectic form $\Omega$ is called a \textit{ multisymplectic manifold}.
\end{defi}
\begin{defi}(\cite{Helein2010}):
 The $(n+1)$-form $\Omega$ is called \textit{pre-multisymplectic form} if it's closed.
 A differential manifold $\mathcal{M}$ equiped of $\Omega$ is called \textit{pre-multisymplectique manifold}.\\
\end{defi}
\subsection{Lagrangian formulation}
\begin{defi}
A 2-form \textit{areolar action} is a 2-form $\ell$ on $Gr^2\mathcal{M}$, which can be written
\[
 \ell(x,y):=\ell_{12}(x,y)dx^1\wedge dx^2+\ell_{23}(x,y)dx^2\wedge dx^3+ \ell_{31}(x,y) dx^3 \wedge dx^1.
\]
where $\ell_{12}(x,y),\ell_{23}(x,y)$ and $\ell_{31}(x,y)$ are homogeneous of degree 0 in $y$ over $]0,+\infty[$.
\end{defi}
\begin{defi}
We define a homogeneous Lagrangian $L:\Lambda^2T\mathcal{M}\rightarrow\mathbb{R}$ as a continuous function of $\Lambda^2T\mathcal{M}$, such that
\begin{enumerate}
\item $L$ is of class $\mathcal{C}^1$ over $\Lambda^2T\mathcal{M}\setminus\sigma_0$ where $\sigma_0$ is the zero section. 
\item $L$ is homogeneous of degree 1 in $y$ over $]0,+\infty[$. 
\end{enumerate}
\end{defi}
\begin{rem}
 If $u_1,u_2$ are two vectors of the element $E_x$ then $L(x,u_1\wedge u_2)$ represent \textit{``the infinitesimal area''} of the parallelepiped formed by $u_1$ and $u_2$.
\end{rem}

Denote by $\Sigma$ an oriented surface in $\mathcal{M}$, let $(u_1,u_2)$ be a moving basis on $\Sigma$ and $(\theta_1,\theta_2)$ its dual basis, we define 
\begin{equation}\label{action1}
  \mathcal{L}(\Sigma):=\int_\Sigma L(x,u_1\wedge u_2)\theta_1\wedge\theta_2,
\end{equation}
Thus, by Euler formula, we have
 $$L(x,y)=\sum_{\alpha<\beta}\frac{\partial L}{\partial y^ {\alpha\beta}}(x,y)y^{\alpha\beta}=\sum_{\alpha<\beta}\ell_{\alpha\beta}(x,y)y^{\alpha\beta}.$$
 Where
\[
 \left\{
\begin{array}{c}
 \frac{\partial L}{\partial y^{12}}=\ell_{12}\\
\frac{\partial L}{\partial y^{23}}=\ell_{23}\\
\frac{\partial L}{\partial y^{31}}=\ell_{31}\\
\end{array}
\right.
\]
So
\[
 \ell(x,y):=\sum_{\alpha<\beta}\left(dx^\alpha\wedge dx^\beta\otimes\frac{\partial} {\partial y^{\alpha\beta}}\right)\pdt\ \ dL
= \frac{\partial L}{\partial y^{12}}dx^1\wedge dx^2+ \frac{\partial L}{\partial y^{23}}dx^2\wedge dx^3+ \frac{\partial L}{\partial y^{31}}dx^3 \wedge dx^1.
\]
We deduce that the inner product by $\sum_{\alpha<\beta}\left(dx^\alpha\wedge dx^\beta\otimes\frac{\partial} {\partial y^{\alpha\beta}}\right)$ is the canonical operator of $(\Lambda^2T_x\mathcal{M})^\ast \rightarrow\Lambda^2T_x^\ast\mathcal{M}$ and
\[
 \ell=\sum_{\alpha<\beta}\left(dx^\alpha\wedge dx^\beta\otimes\frac{\partial} {\partial y^{\alpha\beta}}\right)\pdt\ \ dL
\]
with condition $y^{\alpha\beta}+y^{\beta\alpha}=0$.
\begin{defi}
A surface $\Sigma$ is called a critical point of $\mathcal{L}$, if and only if, for all compact $K\subset\mathcal{M},\ \ K\cap\Sigma$ is a critical point of$$\mathcal{L}_K(\Sigma):=\int_{\pi_{Gr^2\mathcal {M}}^{-1}(K)\cap T\Sigma}\ell$$ where, $T\Sigma=\{(x,T_x\Sigma)\in Gr^2\mathcal {M}; x\in\Sigma\}$.
\end{defi}
The action (\ref{action1}) given,
\[
 \mathcal{L}(\Sigma):=\int_{T\Sigma}\ell=\int_{T\Sigma}\sum_{\alpha<\beta}\left(dx^\alpha\wedge dx^\beta\otimes\frac{\partial} {\partial y^{\alpha\beta}}\right) \pdt\ \ dL.
\]
\subsection{Hamiltonian formulation (multisymplectic)}
 We write the \textit{Legendre transform} function by 
\[
 \Lambda^2T\mathcal{M}\ \ \ \longrightarrow \ \ \ \Lambda^2T^\ast\mathcal{M}\ \ \ \ \ \ \ \ \ \ 
\]
\[
 (x,y)\ \ \ \longmapsto\ \ \ \left(x,\frac{\partial L}{\partial y}(x,y)\right),
\]
where $\frac{\partial L}{\partial y}(x,y)=\left(\frac{\partial L}{\partial y^{\alpha\beta}}(x,y)\right)_{\alpha<\beta}$. We consider a Hamiltonian function by
\[
 \mathcal{H}(x,P)=\sum_{\alpha<\beta}p_{\alpha\beta}y^{\alpha\beta}-L(x,y),
\]
where $y$ is a solution of $\frac{\partial L}{\partial y}(x,y)=p_{\alpha\beta}$. Then, by Euler formula, we have
\[
 L(x,y)=\frac{\partial L}{\partial y^{\alpha\beta}}(x,y)y^{\alpha\beta}=\sum_{\alpha<\beta} p_{\alpha\beta}y^{\alpha\beta} \text{ where } p_{\alpha\beta}=\frac{\partial L}{\partial y^{\alpha\beta}}(x,y),
\]
so the Hamiltonian $\mathcal{H}$ vanishes.
\begin{defi}
We say that the Lagrangian $L$ is nondegenerate if it's of class $\mathcal{C}^2$ over $\Lambda^2T\mathcal{M}\setminus\{\sigma_0\}$ and $\frac{\partial^2(L^2)}{\partial y^{\alpha \beta}\partial y^{\alpha'\beta'}}>0$ out of $\sigma_0$.
In particular $L^2$ is then a function stricly convexe.
\end{defi}

\begin{lem}
If the Lagrangian $L$ is nondegenerate, then,
\[
 \text{rank}\left(\frac{\partial^2 L^2}{\partial y^{\alpha \beta}\partial y^{\alpha'\beta'}}\right)=1+\text{rank}(\text{Hess}(L)).
\]
\end{lem}
\begin{proof}
The Lagrangian $L$ is class $\mathcal{C}^1$ over $\Lambda^2T_x\mathcal{M}$, and class $\mathcal{C}^\infty$ over $\Lambda^2T_x\mathcal{M}\setminus \{\sigma_0\}$, thus $L^2$ is homogeneous of degree 2 and of class $\Lambda^2T_x\mathcal{M}$ and of class $\mathcal{C}^\infty$ over $\Lambda^2T_x\mathcal{M}\setminus \{\sigma_0\}$. For $L$ is nondegenerate, this leads to the convexity. It's easy to check that
 $\frac{\partial(L^2)}{\partial y^{\alpha\beta}}(0)=0$. 
Note $T_yS_x=\ker\left(\frac{\partial L}{\partial y^{\alpha\beta}}(S_x)\right)$, if $(e_1,e_2)$ is a basis of $T_yS_x$ and $e_3=y=y^{\alpha'\beta'}\frac{\partial }{\partial y^{\alpha\beta}}$, we denote by $(z^1,z^2,z^3)$ the coordinates in basis  $(e_1,e_2,e_3)$, we have $\frac{\partial L}{\partial z^1}=\frac{\partial L}{\partial z^2}=0$. More $d\left(\frac{\partial L}{\partial y^{\alpha\beta}}\right) (e_3)= y^{\alpha'\beta'}\frac{\partial }{\partial y^{\alpha'\beta'}} \left(\frac{\partial L}{\partial y^{\alpha\beta}}\right)=\frac{\partial^2 L}{\partial y^{\alpha\beta}\partial y^{\alpha'\beta'}}y^{\alpha'\beta'}=0$ by homogeneity, which results in the coordinates $z^a$ by $\frac{\partial^2 L}{\partial z^3\partial z^a}=0 \ \ \forall a=1,2,3$. Thus if we compute $\frac{\partial^2(L^2)}{\partial y^{\alpha \beta}\partial y^{\alpha'\beta'}}$ out of $\sigma_0$. we find
\begin{equation}
\text{rank}\left(\frac{\partial^2 L^2}{\partial z^a\partial z^b}\right)_{1\leq a,b\leq 3}
=1+\text{rank} d\left(\frac{\partial L}{\partial y}\right)\mid_{TyS}
\end{equation}
In particular $L$ is nondegenerate if and only if $\text{rank}\ \ d\left(\frac{\partial L}{\partial y}\right)\mid_{TyS_x}=2$, so $\frac{\partial L}{\partial y}\mid_{S_x}$ is an immersion $S_x$.
\end{proof}

\begin{prop}\label{surface_convexe}
We consider the Lagrangian $L:\Lambda^2T\mathcal{M}\rightarrow\mathbb{R}$ is of class $\mathcal{C}^k,\ \ k\geq2$. We have the following statement:\\ If $L$ is nondegenerate, then, the image of the Legendre transform,\\ $\mathcal{N}:=\{(x,p)\in\Lambda^2T^\ast\mathcal{M}/p_{\alpha\beta}= \frac{\partial L}{\partial y^{\alpha\beta}}(x,y),$ is a regular convexe hypersurface in $\Lambda^2T^\ast\mathcal{M}$.
\end{prop}
\begin{proof}
For $x$ in $\mathcal{M}$, we define by$$ \mathcal{N}_x=\{p\in\Lambda^2T_x^\ast\mathcal{M}/ p_{\alpha\beta}= \frac{\partial L}{\partial y^{\alpha\beta}}(x,y), y\in\Lambda^2T\mathcal{M}\}.$$
First we show that $\mathcal{N}_x$ is a surface in $\Lambda^2T_x^\ast \mathcal{M}$,
to simplify notations, we denote $L(x,y)=L(y)$. Note
\[
 S_x:=\{y\in\Lambda^2T_x\mathcal{M}\vert L^2(y)=1\}
\]
$S_x$ is a submanifold embedding, in fact, $\forall y\in S$ we have  $y^{\alpha\beta}\frac {\partial L}{\partial y^{\alpha\beta}}(y)=L(y)=1\neq0\Rightarrow d(L^2)_y\neq0$.\\
We know that $\frac{\partial L}{\partial y^{\alpha\beta}}(\lambda y)=\frac{\partial L}{\partial y^{\alpha\beta}}(y)$ thus the image of $\Lambda^2T_x\mathcal{M}\setminus \{\sigma_0\}$ by $\frac{\partial L}{\partial y^{\alpha\beta}}$ is the same as that of $S_x$, because for every half-line, $D\subset\Lambda^2T_x\mathcal{M}\setminus \{\sigma_0\}$, $D$ cut $S_x$ in a single point.\\
We show that $\frac{\partial L}{\partial y^{\alpha\beta}}(S)$ a manifold embedding and convex. For $L^2$ is a strictly convex function, thus the Legendre transform
\[
\frac{\partial(L^2)}{\partial y^{\alpha\beta}}:\Lambda^2T_x\mathcal{M} \ \ \ \longrightarrow\ \ \ \Lambda^2T_x^\ast\mathcal{M}\ \ \ \ \ \ \ \ \ \ \ \ \ \ \ \ \ \ \ \
\]
\[
 y\ \ \ \longmapsto\ \ \ \frac{\partial L^2}{\partial y^{\alpha\beta}}(y)
\]
is a global homeomorphism, for $L^2$ is of class $\mathcal{C}^\infty$ over $\Lambda^2T_x\mathcal{M}\setminus \{\sigma_0\}$ thus
$L^2:\Lambda^2T_x\mathcal{M}\setminus \{\sigma_0\}\longrightarrow\Lambda^2T_x^\ast\mathcal {M}\setminus\{\sigma_0^\ast\}$ is a diffeomorphism, for $S$ is a submanifold embedding, thus $\frac{\partial L^2}{\partial y^{\alpha\beta}}(S_x)$ is so. Or $\frac{\partial(L^2)}{\partial y^{\alpha\beta}}\mid_{S_x}=2L\frac{\partial L}{\partial y^{\alpha\beta}}\mid_{S_x}=2\frac{\partial L}{\partial y^{\alpha\beta}}\mid_{S_x}$
 (because $\forall y\in S_x$ we have $L(y)=1$).
On denote by
\[
 B_x:=\{y\in\Lambda^2T_x\mathcal{M}\vert L^2(y)\leq1\}
\]
Show that $\frac{\partial L}{\partial y^{\alpha\beta}}(B_x)$ is a convex set. Let $P_0,P'_0\in \frac{\partial L}{\partial y^{\alpha\beta}}
(B_x)$. Show that
\[
 \forall t\in[0,1]\Rightarrow tP_0+(1-t)P_0'\in\frac{\partial L}{\partial y^{\alpha\beta}}(B_x)
\]

We have $P_0,P'_0\in \frac{\partial L}{\partial y^{\alpha\beta}}(B_x)$ so there exist $y_0,y'_0\in B_x$ such that $P_0=\frac{\partial L}{\partial y^{\alpha\beta}}(y_0)$ and $P'_0=\frac{\partial L}{\partial y^{\alpha\beta}}(y'_0)$ or $B_x$ 
is convex, thus $ty_0+(1-t)y_0'\in B_x\Rightarrow L(ty_0+(1-t)y_0')=1$ and since $\frac{\partial L}{\partial y^{\alpha\beta}}:\Lambda^2T_x\mathcal{M}\setminus \{\sigma_0\}\longrightarrow\Lambda^2T_x^\ast\mathcal {M}\setminus\{\sigma_0^\ast\}$ is a diffeomorphism, then there exists an unique $P''_0\in\Lambda^2T_x^\ast\mathcal {M}\setminus\{\sigma_0^\ast\}$ (so there exists a unique $y''_0\in\Lambda^2T_x\mathcal{M}\setminus \{\sigma_0\}$) such that $t\frac{\partial L}{\partial y^{\alpha\beta}} (y_0)+(1-t)\frac {\partial L}{\partial y^{\alpha\beta}}(y_0')=P''_0:=\frac {\partial L}{\partial y^{\alpha\beta}}(y''_0)$ and since $L$ homogeneous of degree one, thus, $\forall y\in\Lambda^2T_x\mathcal{M}\setminus \{\sigma_0\}$ it exists an unique $\lambda>0$ such as $L(\lambda y)=1$, so $\lambda=\frac{1}{L(y'')}$ which gives $\frac{1}{L(y'')}y''_0\in B_x$. Then $t\frac{\partial L}
{\partial y^{\alpha\beta}} (y_0)+(1-t)\frac {\partial L}{\partial y^{\alpha\beta}}(y_0')=\frac {\partial L}{\partial 
y^{\alpha\beta}}(y''_0)=\frac {\partial L}{\partial y^{\alpha\beta}}\left(\frac{1}{L(y'')}y''_0\right)\in 
\frac {\partial L}{\partial y^{\alpha\beta}}(B_x)$, which shows the convexity of $\frac {\partial L}{\partial 
y^{\alpha\beta}}(B_x)$ thus $\partial B_x=S_x$ is a convex surface .
\end{proof}

\begin{rem}
If $\frac{\partial^2(L^2)}{\partial y^{\alpha\beta}\partial y^{\alpha'\beta'}} >0$, then $\mathcal{N}_x$ is topologically a sphere in $\Lambda^2T^\ast_x\mathcal{M}$, for all $x\in\mathcal{M}$.
\end{rem}
\subsection{Application}
Let the 2-form $\theta=p_{12}dx^1\wedge dx^2+p_{23}dx^2\wedge dx^3+ p_{31}dx^3\wedge dx^1$ on $\Lambda^2T^\ast \mathcal{M}$, then
\[
 d\theta=dp_{12}\wedge dx^1\wedge dx^2+dp_{23}\wedge dx^2\wedge dx^3+dp_{31}\wedge dx^3\wedge dx^1,
\]
 We have $\theta$ is a multisymplectic form on $\Lambda^2T^\ast\mathcal{M}$, and $d\theta\arrowvert_\mathcal{N}$ is a pre-multisymplectic form.
\begin{thm}
We have $\Sigma\subset\mathcal{M}$ a critical point of $\int_{T\Sigma} L$ if and only if the image of $T\Sigma$ by the Legendre is a surface in $ \mathcal{N} $ such that $\theta\vert_\mathcal{N}=0$, otherwise \\ $\frac{\partial L}{\partial y}(T\Sigma)\subset(\mathcal{N};\theta\mid_\mathcal{N})$, if we denote by
\[
 \beta=\sum_{\alpha<\gamma}p_{\alpha\gamma}dx^\alpha\wedge dx^\gamma=\theta\mid_\mathcal{N}
\]
then,	
\begin{equation}\label{action3}
 \mathcal{L}(\Sigma)=\int_{\frac{\partial L}{\partial y}(T\Sigma)}\beta=\int_{\Gamma}\beta
\end{equation}
 where $\Gamma:=\frac{\partial L}{\partial y}(T\Sigma)\subset\mathcal{N}$.
\end{thm}

\begin{proof}
 Easily we can see that $$\left(\frac{\partial L}{\partial y}\right)^\ast\theta=\frac{\partial L}{\partial y^{12}}dx^1\wedge dx^2+ \frac{\partial L}{\partial y^{23}}dx^2\wedge dx^3+ \frac{\partial L}{\partial y^{31}}dx^3 \wedge dx^1=\ell(x,y),$$ thus
\[
\mathcal{L}(\Sigma)=\int_{T\Sigma}\ell=\int_{T\Sigma}\left(\frac{\partial L}{\partial y}\right)^\ast\theta,
\]
since $\frac{\partial L}{\partial y}\mid_{T\Sigma}$ is an embedding, we get
\[
\mathcal{L}(\Sigma)=\int_{T\Sigma}\left(\frac{\partial L}{\partial y}\right)^\ast\theta=\int_{\frac{\partial L}{\partial y}(T\Sigma)}\theta=\int_{\Gamma}\theta.
\]
\end{proof}




\subsection{Comparison with graphs}
Let $\mathcal{M}=\mathbb{R}^2\times\mathbb{R}$ and $\Sigma$ be the graph of a smooth function $f:\mathbb{R}^2\rightarrow\mathbb{R}$. \\ For $(x,y)=(x^1,x^2,f(x^1,x^2))\in\Sigma$, the tangent bundle $T_{(x,y)}\Sigma$ is generated by
\[
u_1=\left(
\begin{array}{c}
1\\
0\\
\frac{\partial f}{\partial x^1}
\end{array}
\right)
\text{ and }
u_2=\left(
\begin{array}{c}
0\\
1\\
\frac{\partial f}{\partial x^2}
\end{array}
\right).\ \ \ \ \ \ \ \ \ \ \ \
\]
We have
\[
 u_1\wedge u_2=\left(\frac{\partial}{\partial x^1}+\frac{\partial f}{\partial x^1} \frac{\partial}{\partial x^3}\right)\wedge\left(\frac{\partial}{\partial x^2} +\frac{\partial f}{\partial x^2} \frac{\partial}{\partial x^3}\right)\ \ \ \ \ \ \ \ \ \ \ \ \ \ \ \ \ \ \ \ \ \ \ \ \ \ \ \ \ \ \
\]
\[
 =\frac{\partial}{\partial x^1}\wedge\frac{\partial}{\partial x^2}-\frac{\partial f}{\partial x^1}\left(\frac{\partial}{\partial x^2}\wedge\frac{\partial}{\partial x^3}\right)-\frac{\partial f}{\partial x^2}\left(\frac{\partial}{\partial x^3}\wedge\frac{\partial}{\partial x^1}\right).
\]
Denote $y^{12}=1, \ \ y^{23}=-\frac{\partial f}{\partial x^1}\text{ and } y^{31}=-\frac{\partial f}{\partial x^2}$. Let $F(x^1,x^2,f,df)$ be a Lagrangian density such that the action over an infinitesimal square $dx^1\wedge dx^2$ is $F(x^1,x^2,f,df)dx^1\wedge dx^2$. Then there exist a 2-form areolar action $\ell$, such that for any surface $\Sigma\subset\mathbb{R}^3$, we have
\[
 \int_{\Omega}F(x^1,x^2,f,df)dx^1\wedge dx^2=\mathcal{L}(\Sigma)=\int_{T\Sigma}\ell.
\]
 Indeed, if $(e_1,e_2)$ is a basis of $\mathbb{R}^2$ and $(u_1,u_2)$ is a basis of $T_x\Sigma$, then the infinitesimal action of $F$ on the parallelogram generated by $(e_1,e_2)$ is
\[
 F(x^1,x^2,f,df)dx^1\wedge dx^2(e_1,e_2)= F(x^1,x^2,f,df)dx^1\wedge dx^2(u_1,u_2)= F(x^1,x^2,f,df)y^{12}.
\]
If we want
\[
 \sum_{1\leq\alpha<\beta\leq3}\ell_{\alpha\beta}(x,y)dx^\alpha\wedge dx^\beta(u_1,u_2) =\sum_{1\leq\alpha<\beta\leq3}\ell_{\alpha\beta}(x,y)y^{\alpha\beta}=L(x,y),
\]
then, we have
 \[
L(x,y)= \sum_{1\leq\alpha<\beta\leq3}\frac{\partial L}{\partial y^{\alpha\beta}}y^ {\alpha\beta}=\sum_{1\leq\alpha<\beta\leq3}\ell_{\alpha\beta}y^ {\alpha\beta}=y^{12} F(x^1,x^2,f,df).
\]
Since, $\frac{y^{23}}{y^{12}}=-\frac{\partial f}{\partial x^1},\ \ \frac{y^{31}} {y^{12}}=-\frac{\partial f}{\partial x^2}$, then
\[
L(x,y)=y^{12}F\left(x,f,-\frac{y^{23}}{y^{12}},-\frac{y^{31}}{y^{12}}\right).
\]
\begin{rem}
 The generalization of the above formula to hypersurfaces in a smooth manifold $\mathcal{M}$ of dimension $n$ yields
\[
L(x_1...x_n,[y])=y^{1...n-1}F\left(x,f,-\frac{y^{2...n}}{y^{1...n-1}},...,-\frac{y^{n1...n-2}}{y^{1...n-1}}\right).
\]
\end{rem}

\section{Formulation of the problem with a smooth submanifold of dimension $ p $}
Let $\mathcal{M}$ be a smooth manifold of dimension $n\in\mathbb{N}^\ast$, for all $x=(x_1,...,x_n)\in\mathcal{M}$ and $p<n$, then a basis of $\Lambda^pT_x\mathcal{M}$ is 
\[
\left(\frac {\partial}{\partial x^{\imath_1}}\wedge...\wedge\frac{\partial}{\partial x^{\imath_p}}\right)_{1\leq\imath_1<...<\imath_p\leq n}.
\]
Note $(x,y)=(x_1,...,x_n,y^{\imath_1...\imath_p})_{1\leq\imath_1<...<\imath_p \leq n}$ the coordinates on $\Lambda^p_DT_x\mathcal{M}\setminus\{\sigma_0\}$. Let $E$ an oriented element\footnote{The element is a vector subspace of $T_x\mathcal{M}$.} in the tangent space $T_x\mathcal{M}$. \\ If $(u_1,...,u_p)$ a direct basis of $E$ then we can assume that $u_1\wedge...\wedge u_p\in\Lambda^pT_x\mathcal{M}\setminus\{\sigma_0\}$ represents $E$, denote 
\[
 \Lambda^p_DT_x\mathcal{M}:=\{u_1\wedge...\wedge u_p|(u_1,...,u_p)\text{ direct basis of $E$, and oriented in } T_x\mathcal{M}\},
\]
 thus for $u\in \Lambda^p_DT_x\mathcal{M}$ we have
\[
 u:=u_1\wedge...\wedge u_p=\sum_{1\leq\imath_1<...<\imath_p\leq n}y^{\imath_1...\imath_p} \frac {\partial}{\partial x^{\imath_1}}\wedge...\wedge\frac{\partial}{\partial x^{\imath_p}}.
\]
\begin{defi}
For all $x\in\mathcal{M}$, we define the Grassmannian bundle by:
\[
 Gr_x^p\mathcal{M}:=\{(x,E)\vert x\in\mathcal{M}, \ \ E \text{ an oriented element in } T_x\mathcal{M}\}
\]
\end{defi}
\begin{rem} 
We have
\[
 Gr_x^p\mathcal{M}\simeq(\Lambda^p_DT_x\mathcal{M}-\{\sigma_0\})/\mathbb{R}^\ast
\]
Then the homogeneous coordinates on the Grassmannian bundle over $\mathcal{M}$ of dimension $p<n$, are $(x_1,...,x_n;[y^{\imath_1...\imath_p}]_{1\leq\imath_1<...<\imath_p\leq n})$ thus,
 \[
  Gr^p\mathcal{M}:=\{(x;[y]):=(x_1,...,x_n;[y^{\imath_1...\imath_p}]_{1\leq\imath_1<...<\imath_p\leq n})\vert x=(x^1,...,x^n)\in\mathcal{M}
\]
\[
 E=(y)=[y^{\imath_1...\imath_p}]_{1\leq\imath_1<...<\imath_p\leq n}\text{ oriented element in }T_x\mathcal{M}\}
\]
with condition $y^{\imath_1...\imath_k...\imath_{k'}...\imath_p}+y^{\imath_1...\imath_ {k'}...\imath_k...\imath_p}=0$. 
\end{rem}
\subsection{Lagrangian formulation}
\begin{defi}
Let $E_x=(x,[y]):=(x_1,...,x_n;[y^{\imath_1...\imath_p}]_{1\leq\imath_1<...<\imath_p\leq n})\in Gr^p\mathcal{M}$, we define a $p$-form areolar action $\ell$ on $Gr_x^p\mathcal{M}$ by
\[
 \ell(x,[y]):=\sum_{1\leq\imath_1<...<\imath_p\leq n}\ell_{\imath_1...\imath_p} (x,y)dx_{\imath_1}\wedge ...\wedge dx_{\imath_p},
\]
where $\ell_{\imath_1...\imath_p}$ are homogenous of degree 0 in $y$ over $]0,+\infty[$.
\begin{defi}
 Let a Lagrangian $L:\Lambda^p_DT_x\mathcal{M}\rightarrow \mathbb{R}$ as a smooth function on $\Lambda^p_DT_x\mathcal{M}$ such as
\begin{enumerate}
\item $L$ is $\mathcal{C}^{\infty}$ over $L:\Lambda^pT_x\mathcal{M}\setminus\{\sigma_0\}$ where $\sigma_0$ is the vanish section.
\item $L$ is homogenous of degree 1 in $y$ on $]0,+\infty[$.
\end{enumerate}
\end{defi}
\begin{rem}
 For all $x\in\mathcal{M}$, if $(u_1,...,u_p)$  is a direct basis $E_x\in Gr_x^p\mathcal{M}$, then $L(x,u_1\wedge...\wedge u_p)$ represent \textit{``
the infinitesimal volume action''} formed by $p$ vectors $u_1,...,u_p$.
\end{rem}
Let $(\theta_1,...,\theta_p)$ be the dual basis of $(u_1,...,u_p)$, we define
\begin{equation}\label{action2}
 \mathcal{L}(Gr_x^p\mathcal{M}):=\int_{TGr_x^p\mathcal{M}} L(x,u_1\wedge...\wedge u_p)\theta_1\wedge... \wedge\theta_p.
 \end{equation}
\end{defi}
If we have $\ell_{\imath_1...\imath_p}=\frac{\partial L}{\partial y^{\imath_1...\imath_p}} \text{ with } 1\leq\imath_1<...<\imath_p\leq n$, then for  the Euler formula we have $\sum_{1\leq\imath_1<...<\imath_p\leq n}\frac{\partial L} {\partial y^{\imath_1...\imath_p}} dx_{\imath_1}\wedge ...\wedge dx_{\imath_p}=\sum_ {1\leq\imath_1<...<\imath_p\leq n}\ell_{\imath_1...\imath_p} dx_{\imath_1}\wedge ...\wedge dx_{\imath_p}$ thus
\[
 \ell(x;[y])=\sum_{1\leq\imath_1<...<\imath_p\leq n}\left(dx_{\imath_1}\wedge ...\wedge dx_{\imath_p}\otimes\frac{\partial}{\partial y^{\imath_1...\imath_p}}\right)\pdt\ \ dL
\]
\[
=\sum_{1\leq\imath_1<...<\imath_p\leq n}\frac{\partial L}{\partial y^{\imath_1...\imath_p}} dx_{\imath_1}\wedge ...\wedge dx_{\imath_p}.\ \ \ \ \ \ \
\]
Then $\sum_{1\leq\imath_1<...<\imath_p\leq n}\left(dx_{\imath_1}\wedge ...\wedge dx_{\imath _p}\otimes \frac{\partial}{\partial y^{\imath_1...\imath_p}} \right)$ is the canonical operator $(\Lambda^p_xT\mathcal{M})^\ast\longrightarrow\Lambda^pT_x^\ast \mathcal{M}$, thus we write
\[
 \ell(x;[y])=\sum_{1\leq\imath_1<...<\imath_p\leq n}\left(dx_{\imath_1}\wedge ...\wedge dx_{\imath_p}\otimes\frac{\partial}{\partial y^{\imath_1...\imath_p}}\right)\pdt\ \ dL
\]
Let $L$ be a Lagrangian, thus we can write the action (\ref{action2}),
\[
 \mathcal{L}(Gr^p_x\mathcal{M}):=\int_{TGr^p_x\mathcal{M}}\ell
=\int_{TGr^p_x\mathcal{M}}\sum_{1\leq\imath_1<...<\imath_p\leq n}\left(dx_{\imath_1}\wedge ...\wedge dx_{\imath_p}\otimes\frac{\partial}{\partial y^{\imath_1...\imath_p}}\right)\pdt\ \ dL.
\]
\begin{defi}
A critical point of $\mathcal{L}$ if and only if, for all compact $K\subset\mathcal{M},\ \ K\cap Gr^p_x\mathcal{M}$ is a critical point of $\mathcal{L}_K(Gr^p_x\mathcal{M}):=\int_{\pi^{-1}_{TGr^p\mathcal{M}}(K)\cap TGr^p_x\mathcal{M}}\ell$, where $\pi_{Gr^p\mathcal{M}}:Gr^p\mathcal{M}\longrightarrow\mathcal{M}$ is the canonical projection.
\end{defi}
\subsection{Legendre transform-Hamiltonian formulation (multisymplectic)}
 We consider the Hamiltonian function
 \[
 \mathcal{H}(x,P)=\sum_{1\leq\imath_1<...<\imath_p\leq n}p_{\imath_1...\imath_p} y^{\imath_1...\imath_p}-L(x,y)
\]
where $\frac{\partial L}{\partial y}(x,[y])=\left(\frac{\partial L}{\partial y^{\imath_1...\imath_p}}(x,y)\right)_{\imath_1<...<\imath_p}$ is the Legendre transform and $y$ is a solution of $\frac{\partial L}{\partial y}(x,[y])=p_{\imath_1...\imath_p}(x,[y]) $.\\ Since $L$ is homogenous of degree 1 in $y$ on $]0,+\infty[$, then by formula Euler  we have
\[
 L(x,[y])=\frac{\partial L}{\partial y^{\imath_1...\imath_p}}(x,y) y^{\imath_1...\imath_p} =\sum_{1\leq\imath_1<...<\imath_p\leq n} p_{\imath_1...\imath_p}y^{\imath_1...\imath_p}
\]
we conclude that the Hamiltonian $\mathcal{H}$ vanishes.
\begin{prop} We consider $L:\Lambda^p_D T\mathcal{M}\rightarrow\mathbb{R}$, a Lagrangian nondegenerate, then
 the image of Legendre transform $\mathcal{I}:=\{(x,p)\in\Lambda^p_DT^\ast\mathcal{M}/ p_{\imath_1...\imath_p} (x,[y])= \frac{\partial L}{\partial y^{\imath_1...\imath_p}}(x,[y]), y\in\Lambda^p_D T\mathcal{M}\}$ is a regular hypersurface and convex in $\Lambda^p_DT^\ast \mathcal{M}$.
\end{prop}
\begin{proof}
 Same proof of (\ref{surface_convexe}).
\end{proof}
On $\Lambda^pT^\ast\mathcal{M}$, we define the $p$-form $\theta$ by
\[
 \theta=\sum_{1\leq\imath_1<...<\imath_p\leq n}p_{\imath_1...\imath_p}dx_{\imath_1} \wedge...\wedge dx_{\imath_p},
\]
\begin{rem}
 $\theta$ is a multisymplectic form on $\Lambda^P_DT^\ast\mathcal{M}$ and $d\theta\arrowvert_\mathcal{I}$ is a pre-multisymplectic form.
\end{rem}
\begin{thm}
A critical point of $\int_{TGr_x^p\mathcal{M}} L$ if and only if the image by the Legendre of $TGr_x^p\mathcal{M}$ is a vector subspace in $ \mathcal{M}$, such that $\theta\mid_\mathcal{I}$ is vanishes.
If
\[
 \beta=\theta\mid_\mathcal{I}\text{ and } \Gamma:=\frac{\partial L}{\partial y}(TGr_x^p\mathcal{M})\subset\mathcal{I}
\]
then
\begin{equation}\label{action4}
 \mathcal{L}(Gr_x^p\mathcal{M})=\int_{\frac{\partial L}{\partial y}(TGr_x^p\mathcal{M})}\beta=\int_{\Gamma}\beta.
\end{equation}
\end{thm}

\begin{proof}
 See proof of (1.13).
\end{proof}

\subsection{Comparison with graphs}
We will take $\mathcal{M}=\mathbb{R}^p\times\mathbb{R}^{n-p}$ and $\Gamma$ the graph of smooth function $f:\mathbb{R}^P\rightarrow\mathbb{R}^{n-p}$. Denote $\left(q_\imath^\jmath:=\frac{\partial f_\jmath}{\partial x_\imath}\right) _{\substack{1 \leq\imath\leq p\\ 1\leq\jmath\leq n-p}}$, in a neighberhood of a point $x=(x_1,...,x_p,f_1,...,f_{n-p})\in\Gamma$ the tangent bundle $T_x\Gamma$ is generated by
\[
 u^1=\left(
\begin{array}{c}
1\\
0\\
\vdots\\
0\\
\frac{\partial f_1}{\partial x_1}\\
\frac{\partial f_2}{\partial x_1}\\
\vdots\\
\frac{\partial f_{n-p}}{\partial x_1}
\end{array}
\right),
u^2=\left(
\begin{array}{c}
0\\
1\\
\vdots\\
0\\
\frac{\partial f_1}{\partial x_2}\\
\frac{\partial f_2}{\partial x_2}\\
\vdots\\
\frac{\partial f_{n-p}}{\partial x_2}\\
\end{array}
\right),
...,u^p=\left(
\begin{array}{c}
0\\
\vdots\\
0\\
1\\
\frac{\partial f_1}{\partial x_p}\\
\frac{\partial f_2}{\partial x_p}\\
\vdots\\
\frac{\partial f_{n-p}}{\partial x_p}\\
\end{array}
\right)
\]
so that
\[
 u_1\wedge u_2\wedge...\wedge u_p=\left(\frac{\partial}{\partial x_1}+\sum_{\jmath=p+1} ^n\frac{\partial f_1}{\partial x_{\jmath-p}} \frac{\partial}{\partial x_\jmath}\right)\wedge\left(\frac{\partial} {\partial x_2}+\sum_{\jmath=p+1}^n \frac{\partial f_2}{\partial x_{\jmath-p}}\frac{\partial}{\partial x_\jmath} \right)\wedge...\wedge
\]
\[
 \left(\frac{\partial}{\partial x_p}+\sum_{\jmath=p+1}^n\frac{\partial f_p}{\partial x_{\jmath-p}}\frac{\partial} {\partial x_\jmath}\right) =\sum_{1\leq\imath_1<...<\imath_p\leq n}c^{\imath_1...\imath_p}\frac{\partial}{\partial x_{\imath_1}}\wedge...\wedge\frac{\partial}{\partial x_{\imath_p}}
\]
with $c^{1...p}=1$ and $c^{\imath_1...\imath_p}$ (where $\imath_1...\imath_p\neq1...p)$ are functions which can be computed, and which depends on $\left(\frac{\partial f_\jmath}{\partial x_\imath}\right) _{\substack{1 \leq\imath\leq p\\ 1\leq\jmath\leq n-p}}$. In particular, we can show that  
\[
\left\{
\begin{array}{c}
c^{12..p}=1 \ \ \ \ \ \ \ \ \ \ \ \ \ \ \ \ \ \ \ \ \ \ \ \ \ \ \ \ \ \ \ \ \ \ \ \ \ \ \ \ \ \ \ \ \ \ \ \ \ \ \ \ \ \ \\
\frac{c^{1...\imath-1\imath+1...pp+\jmath}}{c^{123...p}}=\frac{\partial f_\jmath}{\partial x_\imath}\text{ for }1 \leq\imath\leq p,\ \ 1\leq\jmath\leq n-p
\end{array}
\right.
\]
 We suppose $y^{\imath_1...\imath_p}=c^{\imath_1...\imath_p}$, $(e_1,...,e_p,e_{p+1},...,e_n)$ is a basis of $\mathbb{R}^n\times\mathbb{R}^{n-p}$ and $(u_1,...,u_p)$ a basis of $T_x \Gamma$. \\ If $F(x_1,...,x_p,f_1,...,f_{n-p},\nabla f)$ is a Lagrangian density, then its infinitesimal  action  over the parallelepiped generated by $(e_1,...,e_p)$ is
\[
  F(x_1,...,x_p,f_1,...,f_{n-p},\nabla f)dx_1\wedge...\wedge dx_p(e_1,...,e_p)=
\]
\[ 
F(x_1,...,x_p,f_1,...,f_{n-p},\nabla f)dx_1\wedge...\wedge dx_p(u_1,...,u_p)=  F(x_1,...,x_p,f_1,...,f_{n-p},\nabla f)y^{1...p}
\]
But we ask that this precedent egality is equivalent to
\[
 \sum_{\substack{1\leq\imath\leq p\\ 1\leq\jmath\leq n-p}} \ell_{1...\imath-1\imath+1...p p+\jmath}(x,[y])dx_1\wedge...\wedge dx_{\imath-1}\wedge dx_{\imath+1}...\wedge dx_p\wedge dx_{p+\jmath} (u_1,...,u_p)
\]
\[
 \sum_{\substack{1\leq\imath\leq p\\ 1\leq\jmath\leq n-p}} \ell_{1...\imath-1\imath+1...p p+\jmath}(x,[y])y^{1...\imath-1\imath+1...p p+\jmath}=L(x,[y^{1...\imath-1\imath+1...p p+\jmath}])
\]
then
\[
 L(x;[y])=\sum_{\substack{1\leq\imath\leq p\\ 1\leq\jmath\leq n-p}}\frac{\partial L}{\partial y^{1...\imath-1\imath+1...p p+\jmath}}y^{1...\imath-1\imath+1...p p+\jmath}
\]
\[
 =\sum_{\substack{1\leq\imath\leq p\\ 1\leq\jmath\leq n-p}}\ell_{1...\imath-1\imath+1...p p+\jmath}(x;[y]) y^{1...\imath-1\imath+1...p p+\jmath}=y^{1...p}F(x^1,...,x^p,f_1,...,f_{n-p},\nabla f)
\]
thus
\[
 L(x,[y])=y^{1...p}F\left(x^1,...,x^p,f_1,...,f_{n-p},\left(\frac{y^{1...\imath-1\imath+1...pp+\jmath}}{y^{123...p}}\right)_{\substack{1 \leq\imath\leq p\\ 1\leq\jmath\leq n-p}}\right).
\]
\newpage
\nocite{*}
\bibliographystyle{plain}
\bibliography{referencesarticle2}

\end{document}